\documentclass[reqno]{amsart}

\usepackage{amsmath}
\usepackage{amssymb}
\usepackage[margin=1in]{geometry}
\usepackage{xcolor}
\usepackage{comment}

\DeclareMathOperator{\dist}{dist}

\DeclareMathOperator{\diam}{diam}
\DeclareMathOperator{\tr}{tr}

\newcommand{\norm}[1]{\left\lVert#1\right\rVert}

\theoremstyle{plain}
\newtheorem{thm}{Theorem}[section]
\newtheorem{lem}{Lemma}[section]
\newtheorem{prop}{Proposition}[section]
\newtheorem{cor}{Corollary}[section]

\theoremstyle{definition}
\newtheorem{defn}{Definition}[section]

\theoremstyle{remark}

\newtheorem{rmk}{Remark}[section]

\usepackage{hyperref}
\hypersetup{ colorlinks = true, urlcolor = blue, linkcolor = blue, citecolor = red }
\usepackage{enumitem}
\usepackage{nameref}
\makeatletter
\let\orgdescriptionlabel\descriptionlabel
\renewcommand*{\descriptionlabel}[1]{%
  \let\orglabel\label
  \let\label\@gobble
  \phantomsection
  \edef\@currentlabel{#1\unskip}%
  \let\label\orglabel
  \orgdescriptionlabel{#1}%
}
\numberwithin{equation}{section}

\begin{document}

\title[$C^{1,\alpha}$ regularity for degenerate fully nonlinear elliptic equations with oblique boundary]{$C^{1,\alpha}$ regularity for degenerate fully nonlinear elliptic equations with oblique boundary conditions on $C^1$ domains}

\author[Byun]{Sun-Sig Byun}
\address{Department of Mathematical Sciences and Research Institute of Mathematics,
    Seoul National University, Seoul 08826, Republic of Korea}
\email{byun@snu.ac.kr}

\author[Kim]{Hongsoo Kim}
\address{Department of Mathematical Sciences, Seoul National University, Seoul 08826, Republic of Korea}
\email{rlaghdtn98@snu.ac.kr}

\author[Oh]{Jehan Oh}
\address{Department of Mathematics, Kyungpook National University,
	Daegu 41566, Republic of Korea}
\email{jehan.oh@knu.ac.kr}

\thanks {S.-S. Byun was supported by the National Research Foundation of
Korea(NRF) grant funded by the Korea government [Grant No.
2022R1A2C1009312]. H. Kim was supported by the National Research
Foundation of Korea(NRF) grant funded by the Korea government [Grant
No. 2021R1A4A1027378]. J. Oh was supported by the National Research Foundation of Korea(NRF) grant funded by the Korea government [Grant Nos. 2020R1C1C1A01014904 and RS-2023-00217116].}

\makeatletter
\@namedef{subjclassname@2020}{\textup{2020} Mathematics Subject Classification}
\makeatother

\everymath{\displaystyle}

\begin{abstract}
We provide a sharp $C^{1,\alpha}$ estimate up to the boundary for a
viscosity solution of a degenerate fully nonlinear elliptic equation
with the oblique boundary condition on a $C^1$ domain. To this end,
we first obtain a uniform boundary H{\"o}lder estimate with the
oblique boundary condition in an ``almost $C^1$-flat" domain for the
equations which is uniformly elliptic only where the gradient is far from
some point, and then we establish a desired $C^{1,\alpha}$
regularity based on perturbation and compactness arguments.
\end{abstract}

\maketitle

\section{Introduction} \label{sec1}
This paper is concerned with the boundary $C^{1,\alpha}$ regularity
for a viscosity solution of the degenerate fully nonlinear elliptic
equation with oblique boundary condition
\begin{align} \label{equ}
\begin{cases}
|Du|^\gamma F(D^2u) = f \ & \text{ in } \Omega \\
\beta \cdot Du = g & \text{ on } \partial \Omega,
\end{cases}
\end{align}
where $\Omega \subset \mathbb{R}^n$ is a bounded $C^1$ domain with
boundary $\partial \Omega$, $F=F(M)$ is uniformly elliptic with
$F(0)=0$, $\gamma \geq 0$, $f \in C(\Omega)$, $g \in
C^\alpha(\partial \Omega)$ and $\beta \in C^\alpha(\partial \Omega)$
is a vector-valued function satisfying the oblique condition below
in \eqref{bcond}.

This type of a singular/degenerate fully nonlinear equation was
studied first by Birindelli and Demengel in the pioneering works
\cite{Birindelli1, Birindelli2} for the singular case that
$-1<\gamma <0$. For the degenerate case that $\gamma >0$, the
groundbreaking work was made by Imbert and Silvestre in
\cite{Imbert13} where they showed an interior $C^{1,\alpha}$
estimate for a solution of a degenerate fully nonlinear equation. In
\cite{Teixeira15} Ara\'{u}jo, Ricarte and Teixeira proved an optimal interior $C^{1,\alpha}$ estimate when $F$ is convex. For the related interior regularity
results, we refer to \cite{Baasandorj24, Andrade22, Yun24,
Teixeira20, Fang21, Filippis21} and the references therein.

For the Dirichlet boundary condition there has been much progress in
the regularity of solutions of degenerate equations up to the
boundary. Birindelli and Demengel \cite{Birindelli3} proved a global
$C^{1,\alpha}$ regularity under the regular boundary datum.
Ara\'{u}jo and Sirakov \cite{Araujo23} proved a sharp boundary
$C^{1,\alpha}$ regularity under the $C^{1,\alpha}$ boundary datum
and a $C^2$ domain, and they also obtained an optimal global
$C^{1,\alpha}$ regularity when $F$ is convex. D. Li and X. Li
\cite{Li23} proved $C^{1,\alpha}$ regularity on a $C^{1,\alpha}$
domain without using the flattening argument. For a further
discussion on the regularity for the Dirichlet boundary condition,
see \cite{Baasandorj242, Silva23} and the references therein.

For the Neumann boundary condition Milakis and Silvestre
\cite{Silvestre06} obtained regularity results for the uniformly elliptic case that $\gamma =0$
including $C^{1,\alpha}$ estimates on a flat domain. For the oblique
boundary condition D. Li and K. Zhang \cite{Li18} obtained
regularity estimates results such as $C^{1,\alpha}$ on a $C^1$
domain for the uniformly elliptic case. Further notable regularity results regarding
the Neumann and oblique boundary condition are to be
found in for uniformly elliptic equation, see \cite{Bessa23,
Bessa24, Bessa242, Han20, Han22}, etc.

For the singular/degenerate case Patrizi \cite{Patrizi08} proved a
global $C^{1,\alpha}$ regularity for the singular case under the
homogeneous Neumann condition and on a $C^2$ domain. Birindelli,
Demengel and Leoni \cite{Birindelli4} proved the existence and
uniqueness, and then a global H{\"o}lder regularity for the
singular/degenerate case under the mixed boundary condition.
Banerjee and Verma \cite{Banerjee22} established a global
$C^{1,\alpha}$ regularity for the degenerate case under the
nonhomogeneous Neumann condition and on a $C^2$ domain. Ricarte
\cite{Ricarte20} proved an optimal $C^{1,\alpha}$ regularity under a
$C^2$ domain when $F$ is convex.

Our purpose in this paper is to obtain an optimal $C^{1,\alpha}$
regularity for the degenerate problem \eqref{equ}. Comparing with
the earlier papers mentioned above, there are two main points. One
is to consider the oblique boundary condition which is a natural
extension of the Neumann boundary condition. The other is to relax a
regularity requirement on the boundary of the domain, namely from
$C^2$ to $C^1$. To this end, we are going to show that a solution
$u$ under consideration can be approximated by an affine function
with an error of order $r^{1+\alpha}$ in any ball of the radius $r$.
Our proof is based on the iterative argument to find a sequence of
affine function $l_k$ with $u-l_k$ getting smaller in the smaller
ball. This can be made by the scaling and compactness argument
applied to $v= u-l$ with $l(x) = a-q\cdot x$ being an affine
function. However  $u-l$ is not a solution of \eqref{equ}, instead
it is a solution of
\begin{align} \label{equq}
|Dv - q|^\gamma F(D^2v)=f.
\end{align}
Therefore we need to obtain an equicontinuous estimate for a
solution of \eqref{equq} independent of $q$ in order to use the
compactness argument.

In \cite{Imbert13} an interior H{\"o}lder estimate, independent of
$q$, was proved depending on whether $|q|$ is large or $|q|$ is
small. When $|q|$ is small, \eqref{equq} is uniformly elliptic where
the gradient $Dv$ is large, and so the H{\"o}lder estimate can be
shown by the method of sliding cusps in \cite{Imbert16}. On the
other hand when $|q|$ is large, the Lipschitz estimate can be proved
by adapting Ishii-Lion's method in \cite{Ishii92}. For the Neumann
boundary case in \cite{Banerjee22}, a boundary H{\"o}lder estimate
on the flat domain was derived in the spirit of \cite{Imbert13}.
When $|q|$ is small, a boundary H{\"o}lder estimate on the flat domain, for the uniformly elliptic equation that holds only where
the gradient is large, was derived as in \cite{Imbert16} by adapting the method of sliding cusps. When $|q|$ is large, however in general
it seems difficult to use the Ishii-Lion's method for the Neumann
and oblique case. Indeed, there are notable results for the Neumann
and oblique boundary case by using the Ishii-Lion's method as in
\cite{Patrizi08, Birindelli4} where the homogeneous Neumann boundary
condition is treated, and, as in \cite{Barles93, Ishii91} where
additional regularity requirements both on the oblique vector
function $\beta$ and on the boundary of on the domain are added,
respectively. In this regards the authors in \cite{Banerjee22}
instead adapted the approach used by Colombo and Figalli in
\cite{Colombo14}, which was motivated by Savin in the paper
\cite{Savin07}, to use the method of sliding paraboloids for the
establishing of a boundary H{\"o}lder estimate on the flat domain
for the uniformly elliptic equation that holds only where the
gradient is small.

Our analytic tools in the present paper are based on those in the
previous papers \cite{Imbert16}, and, \cite{Banerjee22} where the
standard flattening argument is used as the boundary of the
underlying domain is regular enough, say $C^2$. On the other hand,
since the domain under consideration is allowed to be $C^1$, we
thereby can not employ this flattening argument. More precisely,
unlike \cite{Banerjee22} in which the Neumann boundary condition on
the flat domain is assigned, we here in the present paper are
dealing with the oblique boundary condition on the $C^1$ domain in
order to prove a required boundary H{\"o}lder estimate, independent
of $q$, of a solution of the problem \eqref{equq}. Therefore we
consider a concept of the so called ``almost Neumann" boundary and
``almost $C^1$-flat" domain. ``Almost Neumann" means that
$\sup\left|\frac{\beta'}{\beta_n}\right|$ is small enough so that
$\beta \approx e_n$, and, ``almost $C^1$-flat" domain means that the
$C^1$-norm of the local graph of the boundary of the domain is small
enough so that $\Omega \cap B_1 \approx B^+_1$. This kind of a
perturbation is enough to prove a similar result in \cite{Imbert16} by the sliding cusp method. Moreover the proof of the $L^\epsilon$ estimate
in \cite{Banerjee22} was based on the Calder'{o}n-Zygmund cube
decomposition which can be used for only on the flat domain, whence
we instead use the so called growing ink-spot lemma for a corkscrew
domain (Lemma \ref{ink}). An advantage of this measure covering
lemma is that we only need to consider a ball contained in the
domain so that an interior estimate can be adapted in the ball.
Therefore, we can use the interior measure estimate inside the ball
without considering the boundary, and, by choosing an appropriate
barrier function we can prove a doubling type lemma with the oblique
boundary condition.

Unlike the method in \cite{Imbert16} by Imbert and Silvestre, it seems
difficult to use the method in \cite{Savin07} by Savin for the
oblique boundary condition on a $C^1$ domain. The equation, which is
uniformly elliptic only where the gradient is large, is not scaling
invariant under the scaled function, $v(x) = u(rx)/M$ with $r<1$ and
$M>1$ while the equation, which is uniformly elliptic only where the
gradient is small, is  scaling invariant. In order to deal with this
difficulty, we observe that the equation \eqref{equq} is uniformly
elliptic where $|Du - q|
>1$ and that for the scaling function $v(x) = u(rx)/M$, it is uniformly elliptic
except where $\left|Dv - \frac{r}{M}q\right| <\frac{r}{M}$.
Therefore,  the set where $\left|Dv - \frac{r}{M}q\right|
<\frac{r}{M}$ is disjoint with either the set $|Dv| <A$ or the set
$|Dv|>A+2$ wherever $\frac{r}{M}q$ is, and, it is uniformly elliptic
where $|Dv| <A$ or it is uniformly elliptic where $|Dv|>A+2$.
Consequently, we shall employ the sliding paraboloid method when it
is uniformly elliptic where $|Dv| <A$ while the sliding cusp method
can be used when it is uniformly elliptic where $|Dv|>A+2$. In
summary, if we find a suitable estimate both for a solution
of the uniformly elliptic equation that only holds where $|Du| <A$,
and for a solution of the uniformly elliptic equation that only
holds where $|Du|>A+2$, then the scaled function also satisfies the
same estimate whatever $q$ is. This enables us to prove a H{\"o}lder
estimate for the equation which is uniformly elliptic where $|Du -
q|>1$ independent of $q$.

The paper is organized as follows. In section \ref{sec2} we
introduce basic notations, the main results and preliminaries which
will be used later for the proof of the main theorem. In section
\ref{sec3} we establish a boundary H{\"o}lder estimate with the
oblique boundary condition on an ``almost $C^1$-flat" domain on
which the associated problem is uniformly elliptic where $|Du-q|>1$,
independent of $q$. Section \ref{sec4} is devoted to proving an
improvement of flatness lemma by the compactness method.  In section
\ref{sec5} we finally prove the main theorem, Theorem \ref{main}.

\section{Preliminaries and main results}
\label{sec2}

For $r>0$ we write $B_r = \{|x| < r\}$ the ball with center 0 and
radius $r$, $B^+_r = B_r \cap \{x_n >0\}$ the half-ball, and $T_r =
B_r \cap \{x_n = 0\}$ the flat boundary of $B^+_r$. For $x_0 \in
\mathbb{R}^n$ we write $B_r(x_0) = B_r + x_0$ the ball with center
$x_0$. For a domain $\Omega \subset \mathbb{R}^n$ and $x_0 \in
\mathbb{R}^n$ we write $\Omega_r = \Omega \cap B_r$, $\partial
\Omega_r = \partial\Omega \cap B_r$ and $\Omega_r(x_0) = \Omega \cap
B_r(x_0)$, $\partial\Omega_r(x_0) = \partial\Omega \cap B_r(x_0)$.
We denote by $S(n)$ the space of symmetric $n \times n$ real matrices, and $I$ the identity matrix.

We now list our basic structural assumptions. We always assume that
$\beta \in C^{\alpha}(\partial \Omega)$ is oblique, i.e., there
exists a positive constant $\delta_0>0$ such that
\begin{align}
\label{bcond} \beta \cdot \mathbf{n} \geq \delta_0 \text{ and }
\norm{\beta}_{L^\infty(\partial \Omega)} \leq 1,
\end{align}
where $\mathbf{n}$ is the inner normal of $\partial \Omega$. We
always assume that $F(0)=0$ and $F$ is uniformly elliptic, i.e.,
there exist constants $0< \lambda \leq \Lambda$ such that
\begin{align*}
\lambda\norm{N} \leq F(M+N)-F(M) \leq \Lambda\norm{N}
\end{align*}
for any $M, N \in S(n)$ with $N\geq 0$, where $\norm{N}$ is the
spectral radius of $N$.

Let $\alpha_0 = \alpha_0(n,\lambda,\Lambda,\delta_0) \in (0,1]$ be
the optimal exponent of regularity theory for a homogeneous equation
with constant oblique boundary condition, i.e., any viscosity
solution $h$ of
\begin{align*}
\begin{cases}
F(D^2h)=0 \ & \text{ in } B^+_1 \\
 \beta_0 \cdot Dh = 0 & \text{ on } T_1,
 \end{cases}
\end{align*}
where $\beta_0$ is a constant oblique vector satisfying
\eqref{bcond}, is locally of the class $C^{1,\alpha_0}(B^+_1)$ with
the estimate
\begin{align*}
\norm{h}_{C^{1,\alpha_0}(B^+_{1/2})} \leq
C_e\norm{h}_{L^\infty(B^+_1)}.
\end{align*}
for some constant $C_e = C_e(n,\lambda,\Lambda,\delta_0)>1$.

D. Li and K. Zhang \cite{Li18} proved that there exists a universal
constant $0<\alpha_0<1$ for any uniformly elliptic $F$ and that
$\alpha_0=1$ for any convex $F$.

We now state our main theorem.
\begin{thm}
\label{main} Let $\Omega \in C^1$, $0 \in \partial \Omega_1$,
$\gamma \geq 0$ and $u$ be a viscosity solution of
\begin{align*}
\begin{cases}
|Du|^\gamma F(D^2u) = f \ & \text{ in } \Omega_1 \\
\beta \cdot Du = g & \text{ on } \partial \Omega_1,
\end{cases}
\end{align*}
where $\alpha\in\left(0,\alpha_0\right) \cap \left(0,\frac{1}{1+\gamma}\right]$, $f
\in C(\overline{\Omega})$, and $g, \beta \in C^\alpha(\partial
\Omega)$ with $\beta$ satisfying \eqref{bcond}. Then $u \in
C^{1,\alpha}(\overline{\Omega_{1/2}})$ and
\begin{align*}
\norm{u}_{C^{1,\alpha}(\overline{\Omega_{1/2}})} \leq C \left(
\norm{u}_{L^\infty(\overline{\Omega_1} )} +
\norm{f}_{L^\infty(\overline{\Omega_1})}^{\frac{1}{1+\gamma}} +
\norm{g}_{C^\alpha(\partial \Omega_1)} \right),
\end{align*}
where $C$ depends only on $n,\lambda, \Lambda, \alpha, \gamma,
\delta_0, [\beta]_{C^\alpha(\partial \Omega_1)}$ and $C^1$ modulus
of $\partial \Omega_1$.
\end{thm}

From Theorem \ref{main}, we have the following corollary.

\begin{cor}
Let $\Omega \in C^1$, $0 \in \partial \Omega_1$, $\gamma \geq 0$ and
$u$ be a viscosity solution of
\begin{align*}
\begin{cases}
|Du|^\gamma F(D^2u) = f \ & \text{ in } \Omega_1 \\
\beta \cdot Du + hu = g & \text{ on } \partial \Omega_1,
\end{cases}
\end{align*}
where $\alpha\in\left(0,\alpha_0\right) \cap \left(0,\frac{1}{1+\gamma}\right]$, $f
\in C(\overline{\Omega})$, and $g, \beta, h \in C^\alpha(\partial
\Omega)$ with $\beta$ satisfying \eqref{bcond}. Then $u \in
C^{1,\alpha}(\overline{\Omega_{1/2}})$ and
\begin{align*}
\norm{u}_{C^{1,\alpha}(\overline{\Omega_{1/2}})} \leq C \left(
\norm{u}_{L^\infty(\overline{\Omega_1})} +
\norm{f}_{L^\infty(\overline{\Omega_1})}^{\frac{1}{1+\gamma}} +
\norm{g}_{C^\alpha(\partial \Omega_1)} \right),
\end{align*}
where $C$ depends only on $n,\lambda, \Lambda, \alpha, \gamma,
\delta_0, [\beta]_{C^\alpha(\partial \Omega_1)}$,
$\norm{h}_{C^\alpha(\partial \Omega_1)}$ and $C^1$ modulus of
$\partial \Omega_1$.
\end{cor}

We now recall the definition of viscosity solution. (See
\cite{Ishii92, Caffarelli95})

\begin{defn}
    We say that $u$ is a viscosity subsolution (resp. supersolution) of \eqref{equ} if for any $x_0 \in \Omega \cup \partial \Omega$ and test function $\phi \in C^2(\Omega \cup \partial \Omega)$ such that $u-\phi$ has a local minimum (resp. maximum) at $x_0$, then
    \begin{align*}
        |D \phi(x_0)|^\gamma F(D^2\phi(x_0)) \geq ( \text{resp.}\leq) \ f(x_0)  \quad \text{ if } x_0 \in \Omega,
    \end{align*}
    and
    \begin{align*}
        D \phi(x_0) \cdot \beta(x_0) \geq ( \text{resp.}\leq) \ g(x_0) \quad \text{ if } x_0 \in \partial\Omega.
    \end{align*}
    If $u$ is both subsolution and supersolution, then we call $u$ a viscosity solution.
\end{defn}

Since $\Omega \in C^1$, we see that $\Omega$ satisfies the corkscrew
condition. The following is the definition of the corkscrew
condition \cite{Jerison82}.

\begin{defn}
    For $\rho \in (0,1)$, we say that $\Omega$ satisfies $\rho$-corkscrew condition if for any $x \in \Omega$ and $0<r< \diam(\Omega)/3$, there exists $y \in \Omega$ satisfying $B_{\rho r}(y) \subset B_r(x) \cap \Omega$.
\end{defn}
\begin{rmk}
The standard definition of the (interior) corkscrew condition as in
\cite{Jerison82} is defined only for the boundary point $x \in
\partial\Omega$ unlike the above definition defined for $x \in
\Omega$, but it is easy to show that they are equivalent.
    In fact, assuming $\Omega$ has the corkscrew condition for the boundary, for any $x \in \Omega$ and $0<r< \diam(\Omega)/3$, choose $x_0 \in \partial \Omega$ such that $\dist(x,\partial\Omega) = \dist(x,x_0)$ and if $r/2 <\dist(x,x_0)$, then $B_{r/2}(x) \subset B_r(x) \cap \Omega$.
    If $r/2 \geq \dist(x,x_0)$, then by the definition of the corkscrew condition for the boundary, there exists a ball $B_{\rho r/2}(y) \subset B_{r/2}(x_0) \cap \Omega \subset B_{r}(x) \cap \Omega$, which implies $\Omega$ satisfies the above definition of the corkscrew condition.
\end{rmk}
For the corkscrew domain we can obtain a simple measure covering
lemma as in \cite{Savin07}.
\begin{lem}
\label{ink} Let $\Omega$ satisfy $\rho$-corkscrew condition and
assume that there exist two open sets $E,F \subset \Omega$ and some
constant $\epsilon \in (0,1)$ such that:
    \begin{enumerate}
        \item $E \subset F \subset \Omega$ and $F \neq \Omega$.
        \item For any ball $B \subset \Omega$ satisfies $|B \cap E| > (1-\epsilon)|B|$, then $\tilde{\rho} B \cap \Omega \subset F$ where $\tilde{\rho} = \frac{4}{\rho}$.
    \end{enumerate}
    Then $|E| \leq (1-3^{-n}\rho^n\epsilon )|F|$.
\end{lem}

\begin{proof}
We write $E^c = \Omega \setminus E$ and $F^c = \Omega \setminus F \neq \emptyset$. Given $x \in F$, let $r = \dist(x,F^c)<\diam(\Omega)$.
We claim that
\begin{align*}
|B_{r/3}(x) \cap \Omega \cap E^c| \geq 3^{-n}\rho^n\epsilon
|B_{r}(x)|.
\end{align*}
Applying $\rho$-corkscrew condition to $B_{r/3}(x)$, we know that there exists $y \in \Omega$ such that $B_{\rho r/3}(y) \subset \Omega \cap B_{r/3}(x)$.
Thus, $\dist(y, F^c) \leq \dist(y,x)+\dist(x,F^c) \leq \frac{4-\rho}{3}r<\frac{4}{3}r$, which implies $\frac{4}{\rho} B_{\rho r/3}(y) \cap \Omega \not\subset F$.
Therefore, we have $|B_{\rho r/3}(y) \cap E| \leq (1-\epsilon)|B_{\rho r/3}|$. Thus,
\begin{align*}
|B_{r/3}(x) \cap \Omega \cap E^c| &\geq |B_{\rho r/3}(y) \cap E^c| \\
    &= |B_{\rho r/3}(y)| - |B_{\rho r/3}(y) \cap E|\\
    &\geq \epsilon|B_{\rho r/3}| = 3^{-n}\rho^n\epsilon|B_{r}(x)|.
\end{align*}
Now for every $x \in F$ let $r = \dist(x,F^c)$. Then $\{B_r(x)\}$ is
a covering of $F$. By the Vitali covering lemma, there exists a
subcover $\{B_{r_i}(x_i)\}$ of $F$ such that $\{B_{r_i/3}(x_i)\}$
are disjoint. Considering $B_{r_i/3}(x_i) \cap \Omega \subset F$, we
have
\begin{align*}
    |F| &\leq \sum_{i}|B_{r_i}(x_i)| \\
    &\leq 3^n \rho^{-n}\epsilon^{-1} \sum_{i}|B_{r_i/3}(x_i) \cap \Omega \cap E^c|\\
    &\leq 3^n \rho^{-n}\epsilon^{-1} |F \cap E^c|.
\end{align*}
Therefore, we conclude that
\begin{align*}
    |E| &= |F| - |F\cap E^c| \\
    &\leq (1-3^{-n}\rho^n\epsilon )|F|.
\end{align*}
\end{proof}

In the next section we will consider an ``almost $C^1$-flat" domain
$\Omega_1$. We always assume $0 \in \partial \Omega_1$. Then the
boundary $\partial \Omega_1$ can be represented by a graph of
function $\varphi = \varphi_{\Omega} : T_1 \rightarrow \mathbb{R}$
and $\varphi \in C^1$. This means that we have
\begin{align*}
\{(x',x_n) \in B_1|x_n = \varphi(x')\} = \partial \Omega_1, \
\{(x',x_n) \in B_1|x_n > \varphi(x')\} \subset \Omega_1.
\end{align*}
We also write $\beta \in C^\alpha(T_1)$ by $\beta(x') = \beta(x',\varphi(x'))$.

Note that the inner normal vector of $\partial\Omega$ is $\mathbf{n} = \frac{1}{\sqrt{|D\varphi|^2+1}}(D\varphi,1)$.
Thus if $[\varphi]_{C^1} \leq \delta_0/2$, we have $\beta_n \geq \delta_0/2$ using $\beta \cdot \mathbf{n} \geq \delta_0$.
Instead of \eqref{bcond}, we always assume that $\beta = (\beta',\beta_n)$ and $\varphi$ satisfy
\begin{align} \label{bcond2}
    \varphi(0) = 0, \quad [\varphi]_{C^1} \leq \delta_0/2, \quad \beta_n \geq \delta_0/2, \quad \norm{\beta}_{L^\infty} \leq 1.
\end{align}
Then we can define $\tau \in C(T_1)$ as $\tau(x) :=
\frac{\beta'(x)}{\beta_n(x)}$, which represents how close $\beta$ is
to $e_n$. Note that $\sup|\tau|$ is bounded by $\sup|\tau| \leq
\sqrt{\frac{4}{\delta_0^2}-1}$.

We next show that this ``almost $C^1$-flat" boundary of $\Omega_1$ with
a small $C^1$-norm satisfies the corkscrew condition.

\begin{prop}
\label{cork} Suppose $0 \in \partial \Omega_1$ and $[\varphi]_{C^1}
\leq \mu<1/3$. Then $\Omega_1$ satisfies $(1-3\mu)/4$-corkscrew
condition.
\end{prop}

\begin{proof}
For any $x_0 \in \Omega_1$ and $r<2/3$, let us define
\begin{align*}
    x_1 = \begin{cases}
    x_0 -\frac{r}{2}\frac{x_0}{|x_0|} \quad \text{if} \ |x_0| >\frac{r}{2} \\
    0 \quad \text{if} \ |x_0| \leq\frac{r}{2}.
    \end{cases}
\end{align*}
Then we have $B_{r/2}(x_1) \subset B_r(x_0) \cap B_1$.
We divide the proof into the 3 cases.

Case 1 : $x_1 =0$.
Then since $|\varphi(x')| \leq  \frac{r}{2}\mu$ in $x' \in T_{r/2}(0)$, for $x_2 = \left(\frac{r}{2}\mu + \frac{r}{8} (1-\mu)\right)e_n$, we have $B_{r(1-\mu)/4}(x_2) \subset B_{r/2}(x_1) \cap \Omega_1 \subset B_r(x_0) \cap \Omega_1$.

Case 2 : $x_1 \neq 0$ and $e_n \cdot \frac{x_1}{|x_1|} \geq
\cos\left(\arctan{\frac{1}{\mu}}\right)$. Note that the cone
$\left\{ x \in B_1 : e_n \cdot \frac{x}{|x|} \geq
\cos\left(\arctan{\frac{1}{\mu}}\right)\right\}$ is in $\Omega_1$.
For $x_2 = x_1 + \left(\frac{r}{2}\mu + \frac{r}{8} (1-\mu)\right)e_n$, we have $B_{r(1-\mu)/4}(x_2)
\subset B_{r/2}(x_1) \cap \Omega_1 \subset B_r(x_0) \cap
\Omega_1$.

Case 3 : $x_1 \neq 0$ and $e_n \cdot \frac{x_1}{|x_1|} <
\cos\left(\arctan{\frac{1}{\mu}}\right)$. Since $|x'_0 -x'_1| <r/2$,
we have $|(x_0)_n -(x_1)_n| <\mu r/2$ and $|\varphi(x'_0) -
\varphi(x')| <\mu r$ for any $x' \in T_{r/2}(x'_1)$. Moreover, since
$(x_0)_n \geq \varphi(x'_0)$, we have $\varphi(x') - (x_1)_n \leq
\frac{3}{2} r\mu$. Therefore, for $x_2 = x_1 + \left(\frac{3}{2}r\mu
+ \frac{r}{8} (1-3\mu)\right)e_n$, we get $B_{r(1-3\mu)/4}(x_2)
\subset B_{r/2}(x_1) \cap \Omega_1 \subset B_r(x_0) \cap \Omega_1$.

Thus we conclude that $\Omega_1$ satisfies $(1-3\mu)/4$-corkscrew
condition.
\end{proof}

\section{H{\"o}lder estimates up to the boundary for uniformly equations that hold only where $|Du-q|>1$}
\label{sec3}
We need a uniform H{\"o}lder estimate for a solution of
$|Du-q|^\gamma F(D^2u) =f$. Note that this equation is uniform
elliptic where $|Du-q|>1$, and we need to find an uniform H{\"o}lder
estimate independent of $q$.

For a given $0<\lambda<\Lambda$, let
\begin{align*}
\mathcal{P}^+_{\Lambda,\lambda}(D^2u,Du)= \Lambda \tr D^2u^+ -
\lambda \tr D^2u^- + \Lambda|Du|,\\
\mathcal{P}^-_{\Lambda,\lambda}(D^2u,Du)= \lambda \tr D^2u^+ -
\Lambda \tr D^2u^- -\Lambda|Du|.
\end{align*}
We want to prove the following theorem:

\begin{thm}
\label{holder} Assume $\Omega \in C^1$. Then there exists a small
$\mu=\mu(\delta_0) \leq \delta_0/2$ such that if
$[\varphi_\Omega]_{C^1} \leq \mu$, for any $u \in
C(\overline{\Omega \cap B_1})$ satisfying
\begin{align*}
\begin{cases}
\mathcal{P}^-(D^2u,Du) \leq C_0 &\text{ in } \{|Du-q|>\theta\} \cap \Omega_1 , \\
\mathcal{P}^+(D^2u,Du) \geq -C_0 &\text{ in } \{|Du-q|>\theta\} \cap \Omega_1, \\
\beta \cdot Du = g &\text{ on } \partial \Omega_1, \\
\norm{u}_{L^\infty(\overline{\Omega_1})} \leq 1, \\
\norm{g}_{L^\infty(T_1)} \leq C_0,
\end{cases}
\end{align*}
for some $q \in \mathbb{R}^n$ and $0<\theta \leq 1$, then we have
$u\in C^{\alpha}(\overline{\Omega_{1/2}})$ for some
$\alpha(n,\lambda,\Lambda,\delta_0)>0$.  We also have the estimates
independent of $q$,
    \begin{align*}
    \norm{u}_{C^{\alpha}(\overline{\Omega_{1/2}})} \leq C(n,\lambda,\Lambda,\delta_0, C_0).
    \end{align*}
\end{thm}

\begin{rmk}
Note that by scaling, we can always make $\sup|\tau|$ arbitrary small so that $\beta$ is ``almost Neumann" with $\Lambda/\lambda$ changing in a universal way.
Indeed, for any $R>1$ and $u \in C(\overline{\Omega_1})$ satisfying the assumption of Theorem \ref{holder}, let
\begin{align*}
    \tilde{u}(x',x_n) = u\left(x',\frac{1}{R}x_n\right)
\end{align*}
then $\tilde{u}$ satisfies the following:
\begin{align*}
    D^2u = A \cdot D^2\tilde{u} \cdot A, \quad Du = A \cdot D\tilde{u}.
\end{align*}
where $A \in S(n)$ is a diagonal matrix $A =
\text{diag}(1,...,1,R)$. Note that $|Du| \leq R|D\tilde{u}|$ and
$\{|D\tilde{u}-\tilde{q}| >\theta\}$ is contained in $\{
|Du-q|>\theta\}$ where $\tilde{q} = \left(q',\frac{1}{R}q_n\right)$.
Moreover, since $I \leq A \leq RI$, for any symmetric matrices
$\lambda I \leq B \leq \Lambda I$, we get $\lambda I \leq ABA \leq
R^2\Lambda I$, which implies
$\mathcal{P}^-_{\lambda,\Lambda}(D^2{u},D{u}) \geq
\mathcal{P}^-_{\lambda,R^2\Lambda}(D^2\tilde{u},D\tilde{u})$.
Therefore, $\tilde{u} \in C(\overline{\tilde{\Omega}_1})$ satisfies
the following equations:
\begin{align*}
    \begin{cases}
        \mathcal{P}^-_{\lambda,R^2\Lambda}(D^2\tilde{u},D\tilde{u}) \leq C_0 &\text{ in } \{|D\tilde{u}-\tilde{q}|>\theta\} \cap \tilde{\Omega}_1, \\
        \mathcal{P}^+_{\lambda,R^2\Lambda}(D^2\tilde{u},D\tilde{u}) \geq -C_0 &\text{ in } \{|D\tilde{u}-\tilde{q}|>\theta\} \cap \tilde{\Omega}_1, \\
        \tilde{\beta} \cdot D\tilde{u} = g/R &\text{ in } \partial \tilde{\Omega}_1, \\
    \end{cases}
\end{align*}
where $\tilde{\beta} = \left(\frac{1}{R}\beta',\beta_n\right)$, $\tilde{\Omega} = \{(x',Rx_n) : (x',x_n)\in \Omega\}$ and $\varphi_{\tilde{\Omega}} = R\varphi_\Omega$.

Then we have $|\tilde{\beta}| \leq 1$, and $\tilde{\beta}_n \geq
\delta_0/2$. Moreover, since $\tilde{\tau} =
\frac{\beta'}{R\beta_n}$, we have $\sup|\tilde{\tau}| \leq
\frac{\sup|\tau|}{R}$ and we can make $\sup|\tilde{\tau}|$ smaller
than some universal constant by choosing a large $R=R(\delta_0)$.
Since $[\varphi_{\tilde{\Omega}}]_{C^1} \leq R[\varphi_\Omega]_{C^1}
\leq R\mu$, we may choose a small $\mu=\mu(\delta_0)$ so that $R\mu
< \delta_0 /2$. Then $\varphi_{\tilde{\Omega}}$ and $\tilde{\beta}$
also satisfy the assumption $\eqref{bcond2}$. Note that
$[u]_{C^{\alpha}(\Omega_{\frac{1}{2R}})} \leq
R^\alpha[\tilde{u}]_{C^{\alpha}(\tilde{\Omega}_{1/2})}$. Therefore,
if we prove Theorem \ref{holder} for small $\sup |\tau|$, then we
get $\tilde{u}  \in C^{\alpha}(\tilde{\Omega}_{1/2})$, and by using
standard covering and interior estimates, we can also prove $u \in
C^{\alpha}(\Omega_{1/2})$.
\end{rmk}

\begin{rmk}
For some $A \subset \mathbb{R}^n$, let $u$ satisfy
\begin{align*}
    \mathcal{P}^-(D^2u,Du) \leq 1 \ \text{ in } \{Du \in A\} \cap B_1,
\end{align*}
then for the scaled function $v(x) = \frac{u(rx)}{M}$ for $M>1$ and
$r<1$, we have
\begin{align*}
    \mathcal{P}^-(D^2v,Dv) \leq \frac{r^2}{M} \ \text{ in } \left\{Dv \in \frac{r}{M}A\right\} \cap B_{1/r}.
\end{align*}
If $\frac{r}{M}A \supset A$, then we have $\mathcal{P}^-(D^2v,Dv)
\leq 1$ in $\{Dv \in A\} \cap B_1$, which is the same assumption on
$u$. Thus, if we prove some result with the above assumption on $u$,
then we can prove the same result on the scaled function $v$ and
find more information about $u$.

For example, if $A = \{|x| \geq \theta\}$, which means the
inequality holds where the gradient is large, then $\frac{r}{M}A
= \left\{|x| \geq \frac{r}{M}\theta\right\} \supset A$ and we can use the
scaled function $v$. However, if $A = \{|x| \leq \theta\}$, which
means the inequality holds where the gradient is small, then
$\frac{r}{M}A = \left\{|x| \leq \frac{r}{M}\theta\right\}
\not\supset A$.

In particular, if $A = \{|x-q| \geq \theta\}$ for some $q \in
\mathbb{R}^n$, which means the inequality holds where $\{|Du-q| \geq
\theta\}$, then $\frac{r}{M}A = \left\{\left|x-\frac{r}{M}q\right| \geq
\frac{r}{M}\theta\right\} \not\supset A$, which implies the scaled
function does not satisfy the same assumption of $u$. But we observe
that the center $q$ moved to $\frac{r}{M}q$ and the radius $\theta$
became $\frac{r}{M}\theta$, smaller than $\theta$. From this
observation, we prove that $v$ satisfies the assumption below.

\begin{prop} \label{two}
    For any $q\in \mathbb{R}^n$, $0<\theta \leq 1$ and $u\in C(B_1)$ satisfying
    \begin{align*}
        \mathcal{P}^-(D^2u,Du) \leq 1 \ \text{ in } \{|Du-q|> \theta\} \cap B_1,
    \end{align*}
    the scaled function $v(x) = u(rx)/M$ where $0<r\leq 1$ and $M\geq 1$ satisfies either
    \begin{align*}
        \mathcal{P}^-(D^2v,Dv) \leq 1 \ \text{ in } \{|Dv| \leq A_0\} \cap B_{1/r}, \ \text{ or } \ \mathcal{P}^-(D^2v,Dv) \leq 1 \ \text{ in } \{|Dv| \geq A_0+2\} \cap B_{1/r},
    \end{align*}
    for any $A_0 >0$.
\end{prop}

\begin{proof}
Note that $\mathcal{P}^-(D^2v,Dv) \leq \frac{r^2}{M} \leq 1 \ \text{ in } \left\{\left|Dv-\frac{r}{M}q\right| \geq \frac{r}{M} \theta\right\} \cap B_{1/r}$.

If $\left|\frac{r}{M}q\right| \geq A_0+1$, then $\{|x| \leq A_0\} \subset \left\{\left|x-\frac{r}{M}q\right| \geq \frac{r}{M}\theta\right\}$.
This is because if $|x| \leq A_0$, then $\left|x-\frac{r}{M}q\right| \geq \left|\frac{r}{M}q\right| -|x| \geq 1 \geq \frac{r}{M}\theta$.
Therefore, we have $\mathcal{P}^-(D^2v,Dv) \leq 1 \ \text{ in } \{|Dv| \leq A_0\} \cap B_{1/r}$.

If $\left|\frac{r}{M}q\right| \leq A_0+1$, then also $\{|x| \geq A_0+2\} \subset \left\{\left|x-\frac{r}{M}q\right| \geq \frac{r}{M}\right\}$.
This is because if $|x| \geq A_0+2$, then $\left|x-\frac{r}{M}q\right| \geq |x|-\left|\frac{r}{M}q\right| \geq 1 \geq \frac{r}{M}\theta$.
Therefore, we have $\mathcal{P}^-(D^2v,Dv) \leq 1 \ \text{ in } \{|Dv| \geq A_0+2\} \cap B_{1/r}$.
\end{proof}
Therefore, we instead find some result for both of inequalities, one holds where $\{|Du| \leq A_0\}$ and the other holds where $\{|Du| \geq A_0+2\}$. Then since the scaled function $v$ satisfies either of the inequalities, we can prove the same result on $v$ whatever $q$ is.
\end{rmk}

Now, we prove the interior measure estimate lemma following the idea of \cite{Imbert16}, using a paraboloid instead of a cusp since we first consider the equation which is uniformly elliptic where the gradient is small.
\begin{lem} \label{A00}
    There exist small $\epsilon_0>0$ and large $K>0$, $A_0>0$ such that for any $u \in C(\overline{B_1})$ satisfying
    \begin{align*}
    \begin{cases}
        u \geq 0 \text{ in } B_1,\\
        \mathcal{P}^-(D^2u,Du) \leq 1 \text{ in } \{|Du| \leq A_0\}\cap B_1, \\
        |\{u>K\} \cap B_1| \geq (1-\epsilon_0)|B_1|,
    \end{cases}
    \end{align*}
    we have $u>1$ in $B_{1/4}$.
\end{lem}
\begin{proof}
 First, assume that $u\in C^2(B_1)\cap C(\overline{B_1})$. We prove this lemma by
contradiction. We assume the contrary that for all $\epsilon_0$, $K$
and $A_0$, we can find $u$ such that the above conditions hold but
$u(x_0) \leq 1$ for some $x_0 \in B_{1/4}$. Consider $G = \{u>K\}
\cap B_{1/4}$. Then $|G| \geq |B_{1/4}|-\epsilon_0|B_1| =
(c-\epsilon_0)|B_1|>0$ for small $\epsilon_0 < c=c(n)$.

We slide a paraboloid $\phi(z) = -10|z-x|^2$ with vertex $x \in G$ from below, until it touches the graph of $u$ for the first time at a point $y \in \overline{B_1}$. Then we have
\begin{align} \label{A001}
    u(y) + 10|y-x|^{2} = \inf_{z\in\overline{B_1}}\{u(z) + 10|z-x|^{2}\}.
\end{align}

We claim that $y \in B_1$.
If $y \in \partial B_1$, then since $u \geq 0$ and $x \in B_{1/4}$, we have
\begin{align*}
    u(y) + 10|y-x|^2 > 10\left|1-\frac{1}{4}\right|^2.
\end{align*}
However, since $x_0 \in B_{1/4}$,
\begin{align*}
    u(x_0) + 10|x_0-x|^2 \leq 1+ 10\left|\frac{1}{2}\right|^2 < 10\left|1-\frac{1}{4}\right|^2.
\end{align*}
which contradicts with \eqref{A001}.

Let $K=1+ 10\left|\frac{1}{2}\right|^2$, then $u(y)<K$. Note that we
have
\begin{align*}
    Du(y) &= D\phi(y) = 20(x-y),\\
    D^2u(y) &\geq D^2\phi(y) = -20I.
\end{align*}
Choosing $A_0>40 \geq \sup_{B_1}|D\phi(y)|$, we have $|Du| \leq A_0$
at a touching point $y$. Therefore, we get $|D^2u(y)| \leq
C(\lambda,\Lambda)$ since $\mathcal{P}^-(D^2u,Du) \leq 1$. Observe
that $x= y+\frac{1}{20}Du(y)$, so let $U \subset B_1$ be the set of
touching points $y$ and define a map $m:U \rightarrow G$ by $m(y): =
x$. Then we have $U \subset \{u<K\}\cap B_1$, $m(U) = G$ and
$|Dm(y)| = \left|I+\frac{1}{20}D^2u(y)\right| \leq C$. Therefore,
using the area formula for $m:U \rightarrow G$,
\begin{align*}
    (c-\epsilon_0)|B_1| \leq |G| =\int_{U}|\det Dm(y)|dy \leq C|U| \leq C\epsilon_0|B_1|,
\end{align*}
which is a contradiction if we choose $\epsilon_0$ small enough.

If $u$ is semiconcave, we can repeat the argument as in the proof of Proposition 3.5 in \cite{Imbert16}.
\end{proof}

Note that if $[\varphi_\Omega]_{C^1} \leq  \mu \leq 1/6$, $\Omega_1$ satisfies $\rho$-corkscrew condition with $\rho =1/8$ by Proposition \ref{cork}.
Since the corkscrew condition is scaling invariant, we can say that $\Omega_{\tilde{\rho}+1} =\Omega_{33}$ satisfies $1/8$-corkscrew condition when $[\varphi_\Omega]_{C^1} < \mu \leq 1/6$.

Now we consider the barrier function $b(x) = |x-\sigma e_n|^{-p}$ for some $\sigma \geq 1$.
Then for $x \in B_{\tilde{\rho}+1}(\sigma e_n)$,
\begin{align*}
    \mathcal{P}^-(D^2b, Db) &= \lambda p(p+1)|x-\sigma e_n|^{-p-2} - \Lambda(n-1)p|x-\sigma e_n |^{-p-2} - \Lambda p|x-\sigma e_n|^{-p-1} \\
    &\geq p|x-\sigma e_n|^{-p-2} \geq p(\tilde{\rho}+1)^{-p-2},
\end{align*}
if $p=p(n,\lambda,\Lambda)$ is large enough.
Moreover, for $x \in \partial \Omega_{\tilde{\rho}+1}(\sigma e_n)$, we have $|x_n| \leq (\tilde{\rho}+1)[\varphi_\Omega]_{C^1} \leq (\tilde{\rho}+1)\mu$ and thus
\begin{align*}
    \beta \cdot Db(x) &= -p|x-\sigma e_n|^{-p-2} (x-\sigma e_n)\cdot \beta \\
    &= \beta_n p|x-\sigma e_n|^{-p-2}(\sigma  - x_n - \tau \cdot x') \\
    &\geq \frac{\delta_0}{2} p|x-\sigma e_n|^{-p-2}(1 - (\tilde{\rho}+1)\mu - (\tilde{\rho}+1)\tau_0) \\
    &> \delta_0 p(\tilde{\rho}+1)^{-p-3}\left(\frac{1}{4}\right) =: \eta_0 >0,
\end{align*}
if $\sup|\tau| \leq \tau_0 < 1/(2\tilde{\rho}+2)$ and $\mu < 1/(2\tilde{\rho}+2)$ is small enough.

Using the above barrier function, we prove the doubling type lemma for the ``almost Neumann" condition on ``almost $C^1$-flat" domain as in \cite{Banerjee22}.
\begin{lem} \label{A01}
    There exist small $\tau_0>0$, $\mu >0$, $\eta>0$, and large $K>1$, $A_0>1$ such that if $\sigma \geq 1$, $\sup|\tau| < \tau_0$ and $[\varphi_\Omega]_{C^1} < \mu$, for any $u \in C(\Omega_{\tilde{\rho}+1}(\sigma e_n))$ satisfying
    \begin{align*}
    \begin{cases}
        u \geq 0 \text{ in } \Omega_{\tilde{\rho}+1}(\sigma e_n),\\
        \mathcal{P}^-(D^2u,Du) \leq 1 \text{ in } \{|Du| \leq A_0\}\cap \Omega_{\tilde{\rho}+1}(\sigma e_n), \\
        \beta \cdot Du \leq \eta \text{ on } \partial \Omega_{\tilde{\rho}+1}(\sigma e_n), \\
        u>K \text{ in } B_{1/4}(\sigma e_n),
    \end{cases}
    \end{align*}
    then $u>1$ in $\Omega_{\tilde{\rho}}(\sigma e_n)$.
\end{lem}
This lemma deals with both interior and boundary cases.
If $\sigma$ is large enough so that $B_{\tilde{\rho}+1}(\sigma e_n)\subset \Omega$, then $\partial \Omega_{\tilde{\rho}+1}(\sigma e_n) = \emptyset$ and it is just the interior estimate.

\begin{proof}
Note that $\overline{B_{1/4}(\sigma e_n)} \subset \Omega$ for small enough $\mu>0$.
We compare $u$ with the following barrier function
\begin{align*}
    B(x) = \frac{K}{2 \cdot 4^p}(|x-\sigma e_n|^{-p}-(\tilde{\rho}+1)^{-p}).
\end{align*}
Then for $\tau_0, \mu < 1/(2\tilde{\rho}+2)$ and $\eta<\eta_0$, we have the following properties:
\begin{enumerate}
    \item $B(x) \leq 0$ in $\mathbb{R}^n \setminus B_{\tilde{\rho}+1}(\sigma e_n)$,
    \item $B(x) < K$ for any $x \in  \partial B_{1/4}(\sigma e_n)$,
    \item $\beta \cdot DB >\frac{K}{2 \cdot 4^p}\eta_0$ on $\partial \Omega_{\tilde{\rho}+1}(\sigma e_n)$.
\end{enumerate}
Moreover, choosing a large $K=K(n,\lambda,\Lambda)\geq 2 \cdot 4^p$ and
letting $A_0 =
\sup_{B_{\tilde{\rho}+1}(\sigma e_n) \setminus B_{1/4}(\sigma
e_n)}|DB|$, then we have
\begin{enumerate}
    \item[(4)] $\mathcal{P}^-(D^2B, DB) \geq \frac{K}{2 \cdot 4^p}p(\tilde{\rho}+1)^{-p-2} \geq 2$ in $\{|DB| \leq A_0\}\cap B_{\tilde{\rho}+1}(\sigma e_n) \setminus B_{1/4}(\sigma e_n)$,
    \item[(5)] $B(x)> 1$ in $B_{\tilde{\rho}}(\sigma e_n)$.
\end{enumerate}
Finally we claim that $u \geq B$ in $ \Omega_{\tilde{\rho}+1}(\sigma e_n) \setminus B_{1/4}(\sigma e_n)$. We
assume the contrary that $u-B$ has a negative minimum at $x_0 \in
\overline{ \Omega_{\tilde{\rho}+1}(\sigma e_n) \setminus B_{1/4}(\sigma e_n)}$.
Then from (1) and (2), $x_0$ cannot be on
$\Omega \cap \partial B_{\tilde{\rho}+1}(\sigma e_n)$ or $\partial
B_{1/4}(\sigma e_n)$. If $x_0 \in \partial \Omega \cap
B_{\tilde{\rho}+1}(\sigma e_n)$, then $\beta \cdot DB(x_0) \leq
\eta <\eta_0$, which contradicts with (3). If $x_0$ is in the interior, then
we have $\mathcal{P}^-(D^2B, DB)(x_0) \leq 1$, which also contradicts
with (4).

Therefore we get $u>1$ in $\Omega_{\tilde{\rho}}(\sigma e_n)$ by (5) and prove the lemma.
\end{proof}
Combining the Lemma \ref{A00} and Lemma \ref{A01}, we obtain the following corollary.
\begin{cor} \label{A0}
    There exist small $\tau_0>0$, $\mu>0$, $\eta>0$, $\epsilon_0>0$, and large $K>1$, $A_0>1$ such that if $\sigma \geq 1$, $B_1(\sigma e_n) \subset \Omega$, $\sup|\tau| \leq \tau_0$ and $[\varphi_\Omega]_{C^1} < \mu$, for any $u \in C( \overline{\Omega_{\tilde{\rho}+1}(\sigma e_n)})$ satisfying
    \begin{align*}
    \begin{cases}
        u \geq 0 \text{ in } \Omega_{\tilde{\rho}+1}(\sigma e_n),\\
        \mathcal{P}^-(D^2u,Du) \leq 1 \text{ in } \{|Du| \leq A_0\}\cap \Omega_{\tilde{\rho}+1}(\sigma e_n), \\
        \beta \cdot Du \leq \eta \text{ on } \partial \Omega_{\tilde{\rho}+1}(\sigma e_n), \\
        |\{u>K\} \cap B_1(\sigma e_n)| \geq (1-\epsilon_0)|B_1|,
    \end{cases}
    \end{align*}
    we have $u>1$ in $\Omega_{\tilde{\rho}}(\sigma e_n)$.
\end{cor}
\begin{proof}
    Let $K_1$, $K_2$ and $A_{01}$, $A_{02}$ be the constants from Lemmas \ref{A00} and \ref{A01} respectively.
    We choose $K=K_1 K_2$ and $A_0 = \max \{A_{01}K_2,A_{02}\}$.
    Then since $B_1(\sigma e_n) \subset \Omega$, $v \in C(\overline{B_1})$ defined by $v(x)=u(x+\sigma e_n)/K_2$ satisfies the assumption of Lemma \ref{A00}, therefore $v>1$ in $B_{1/4}$.
    Thus $u$ satisfies the assumption of Lemma \ref{A01}, and we conclude that $u>1$ in $\Omega_{\tilde{\rho}}(\sigma e_n)$.
\end{proof}

Now we consider the equation which is uniformly elliptic where the gradient is large. We again prove the interior measure estimate lemma similar with Lemma \ref{A00} but by sliding cusp method instead of a paraboloid.

\begin{lem} \label{A10}
    For any $A_1>1$, there exist small  $\epsilon_0>0$ and large $K>1$ such that for any $u \in C(\overline{B_1})$ satisfying
    \begin{align*}
    \begin{cases}
        u \geq 0 \text{ in } B_1,\\
        \mathcal{P}^-(D^2u,Du) \leq 1 \text{ in } \{|Du| \geq A_1\}\cap B_1, \\
        |\{u>K\} \cap B_1| \geq (1-\epsilon_0)|B_1|,
    \end{cases}
    \end{align*}
    we have $u>1$ in $B_{1/4}$.
\end{lem}

\begin{proof}
The proof is similar to that of Lemma \ref{A00}. First, assume that
$u\in C^2(B_1)\cap C(\overline{B_1})$ and assume the contrary that
for all $\epsilon_0$ and $K$, we can find $u$ such that the above
conditions hold but $u(x_0) \leq 1$ for some $x_0 \in B_{1/4}$.

Consider $G = \{u>K\} \cap B_{1/4}$. For some $C>1$, we slide a cusp $\phi(z) = -C|z-x|^{1/2}$ with vertex $x \in G$ from below, until it touches the graph of $u$ for the first time at a point $y \in \overline{B_1}$. Then we have
\begin{align*}
    u(y) + C|y-x|^{1/2} = \inf_{z\in\overline{B_1}}\{u(z) + C|z-x|^{1/2}\}.
\end{align*}
We claim that $y \in B_1$. If $C>1$ is large enough, then $y \not\in
\partial B_1$ by using a similar argument in Lemma \ref{A00}.

Moreover, we choose large $C=C(n,A_1)>1$ such that $A_1 \leq
\min_{B_1} |D\phi|$, and $K=K(n,A_1)>1$ such that $u(y)<K$. Then we have $x\neq y$ so that $-C|z-x|^{1/2}$ is differentiable at $z=y$, and $|Du| \geq A_1$ at a touching
point $y$, which implies that we can use the inequality at that point.
The rest of the proof is same to that of Lemma 3.1 in \cite{Imbert16}.
\end{proof}

We again use the barrier function in Lemma \ref{A01} to prove a
doubling type lemma similar to Lemma \ref{A01}.

\begin{lem} \label{A11}
    For any $A_1>1$, there exist small $\tau_0>0$, $\mu >0$, $\eta>0$, and large $K>1$ such that if $\sigma \geq 1$, $\sup|\tau| \leq \tau_0$ and $[\varphi_\Omega]_{C^1} < \mu$, for any $u \in C(\Omega_{\tilde{\rho}+1}(\sigma e_n))$ satisfying
    \begin{align*}
    \begin{cases}
        u \geq 0 \text{ in } \Omega_{\tilde{\rho}+1}(\sigma e_n),\\
        \mathcal{P}^-(D^2u,Du) \leq 1 \text{ in } \{|Du| \geq A_1\}\cap \Omega_{\tilde{\rho}+1}(\sigma e_n), \\
        \beta \cdot Du \leq \eta \text{ on } \partial \Omega_{\tilde{\rho}+1}(\sigma e_n), \\
        u>K \text{ on } B_{1/4}(\sigma e_n),
    \end{cases}
    \end{align*}
    then $u>1$ in $\Omega_{\tilde{\rho}}(\sigma e_n)$.
\end{lem}

\begin{proof}
We consider the same barrier function $B$ in the proof of Lemma \ref{A01}.
\begin{align*}
     B(x) = \frac{K}{2 \cdot 4^p}(|x-\sigma e_n|^{-p}-(\tilde{\rho}+1)^{-p})
\end{align*}
For $\tau_0, \mu < 1/(2\tilde{\rho}+2)$, we choose large $K=K(n,\lambda,\Lambda,A_1)\geq 2 \cdot 4^p$ satisfying $ \inf_{B_{\tilde{\rho}+1}(\sigma e_n) \setminus B_{1/4}(\sigma e_n)}|DB| > A_1$,
\begin{enumerate}
    \item[(4')] $\mathcal{P}^-(D^2B, DB) \geq 2$ in $\{|DB| \geq A_1\}\cap B_{\tilde{\rho}+1}(\sigma e_n) \setminus B_{1/4}(\sigma e_n),$
\end{enumerate}
and (5) in Lemma \ref{A01}.
By using the same argument, we can prove that $u>1$ in $\Omega_{\tilde{\rho}}(\sigma e_n)$.
\end{proof}

\begin{cor} \label{A1}
    For any $A_1>1$, there exist small $\tau_0>0$, $\mu>0$, $\eta>0$, $\epsilon_0>0$, and large $K>1$, such that if $\sigma \geq 1$, $B_1(\sigma e_n) \subset \Omega$ and $\sup|\tau| \leq \tau_0$, for any $u \in C( \overline{\Omega_{\tilde{\rho}+1}(\sigma e_n)})$ satisfying
    \begin{align*}
    \begin{cases}
        u \geq 0 \text{ in } \Omega_{\tilde{\rho}+1}(\sigma e_n),\\
        \mathcal{P}^-(D^2u,Du) \leq 1 \text{ in } \{|Du| \geq A_1\}\cap \Omega_{\tilde{\rho}+1}(\sigma e_n), \\
        \beta \cdot Du \leq \eta \text{ on } \partial \Omega_{\tilde{\rho}+1}(\sigma e_n), \\
        |\{u>K\} \cap B_1(\sigma e_n)| \geq (1-\epsilon_0)|B_1|,
    \end{cases}
    \end{align*}
    we have $u>1$ in $\Omega_{\tilde{\rho}}(\sigma e_n)$.
\end{cor}

\begin{proof}
    Let $K_1$, $K_2$ be the constants from Lemmas \ref{A10} and \ref{A11} respectively.
    We choose $K=K_1 K_2$, then by the same argument above, we can prove the corollary.
\end{proof}
Now, we summarize what we have done so far.

\begin{cor} \label{meas}
    There exist small $\tau_0>0$, $\mu>0$, $\eta>0$, $\epsilon_0>0$, and large $K>1$, $A_0>1$ such that if $\sigma \geq 1$, $B_1(\sigma e_n) \subset \Omega$, $\sup|\tau| \leq \tau_0$ and $[\varphi_\Omega]_{C^1} < \mu$, for any $u \in C( \overline{\Omega_{\tilde{\rho}+1}(\sigma e_n)})$ satisfying either
    \begin{align*}
    \begin{cases}
            u \geq 0 \text{ in } \Omega_{\tilde{\rho}+1}(\sigma e_n),\\
            \mathcal{P}^-(D^2u,Du) \leq 1 \text{ in } \{|Du| \leq A_0\}\cap \Omega_{\tilde{\rho}+1}(\sigma e_n),\\
            \beta \cdot Du \leq \eta \text{ on } \partial \Omega_{\tilde{\rho}+1}(\sigma e_n), \\
            |\{u>K\} \cap B_1(\sigma e_n)| \geq (1-\epsilon_0)|B_1|,
    \end{cases}
    \text{or}
    \begin{cases}
            u \geq 0 \text{ in } \Omega_{\tilde{\rho}+1}(\sigma e_n),\\
            \mathcal{P}^-(D^2u,Du) \leq 1 \text{ in } \{|Du| \geq A_0+2\}\cap \Omega_{\tilde{\rho}+1}(\sigma e_n),\\
            \beta \cdot Du \leq \eta \text{ on } \partial \Omega_{\tilde{\rho}+1}(\sigma e_n), \\
            |\{u>K\} \cap B_1(\sigma e_n)| \geq (1-\epsilon_0)|B_1|,
    \end{cases}
    \end{align*}
    we have $u>1$ in $\Omega_{\tilde{\rho}}(\sigma e_n)$.
\end{cor}
\begin{proof}
    We choose $A_1 = A_0+2$ in the Corollary \ref{A1} and use Corollaries \ref{A0} and \ref{A1}.
\end{proof}

Using the above Lemma, we prove the $L^\epsilon$ estimate of the
``almost Neumann" condition on a ``almost $C^1$-flat" domain.

\begin{thm} \label{Le}
    There exist small $\tau_0>0$, $\mu>0$, $\eta>0$, $\epsilon>0$, and large $C_1>1$ such that if $\sup|\tau| \leq \tau_0$ and $[\varphi_\Omega]_{C^1} < \mu$, for any $q \in \mathbb{R}^n$, $0<\theta\leq 1$ and $u \in C(\overline{\Omega_{\tilde{\rho}+1}})$ satisfying
    \begin{align*}
        \begin{cases}
            u \geq 0 \text{ in } \Omega_{\tilde{\rho}+1},\\
            \mathcal{P}^-(D^2u,Du) \leq 1 \text{ in } \{|Du-q| \geq \theta\}\cap \Omega_{\tilde{\rho}+1},\\
            \beta \cdot Du \leq \eta \text{ on } \partial \Omega_{\tilde{\rho}+1}, \\
            \inf_{\Omega_{1}} u \leq 1,
        \end{cases}
    \end{align*}
    we have
    \begin{align*}
        |\{u>t\} \cap \Omega_1| \leq C_1t^{-\epsilon}, \ t>0.
    \end{align*}
\end{thm}

\begin{proof}
We claim that for $\epsilon_0>0$, $K>1$ as in Corollary \ref{meas},
\begin{align*}
    |\{ u > K^m\} \cap \Omega_1| \leq (1-3^{-n}\rho^n\epsilon_0 )^m |\Omega_1|.
\end{align*}
This can be proved by induction on $m$. For $m=0$, it is trivial.

Assuming that the claim is true for some $m$, we set
\begin{align*}
    E = \{u > K^{m+1}\} \cap \Omega_1, \ F= \{u > K^{m}\} \cap \Omega_1,
\end{align*}
and we prove that $|E| \leq (1-3^{-n}\rho^n\epsilon_0 ) |F|$ using Lemma \ref{ink}.
Note that $E \subset F \subset \Omega_1$ and $F \neq \Omega_1$ since $\inf_{\Omega_{1}} u \leq 1$.
Let $B = B_r(x_0) \subset \Omega_1$ be a ball satisfying $|B \cap E| >(1-\epsilon_0)|B|$, then $r\leq 1$.
We write $\bar{x_0} = (x_0', \varphi(x_0')) \in \partial \Omega_1$ a  projection of $x_0$ to a boundary, then $\sigma := \frac{(x_0 - \bar{x_0})_n}{r} \geq 1$ since $B_r(x_0) \subset \Omega_1$.

Now, we consider the scaled function $v(y) = \frac{1}{K^m}u(\bar{x_0}+ ry)$.
Then by Proposition \ref{two}, $v$ satisfies either following inequalities
\begin{align*}
     \begin{cases}
        v \geq 0 \text{ in } \tilde{\Omega}_{\tilde{\rho}+1}(\sigma e_n),\\
        \mathcal{P}^-(D^2v,Dv) \leq 1 \text{ in } \{|Dv| \leq A_0\}\cap \tilde{\Omega}_{\tilde{\rho}+1}(\sigma e_n),\\
        \tilde{\beta} \cdot Dv  \leq \eta \text{ on } \partial \tilde{\Omega}_{\tilde{\rho}+1}(\sigma e_n),\\
        |\{v>K\} \cap B_1(\sigma e_n)|> (1-\epsilon_0)|B_1|,
    \end{cases}
    \text{or }
    \begin{cases}
        v \geq 0 \text{ in } \tilde{\Omega}_{\tilde{\rho}+1}(\sigma e_n),\\
        \mathcal{P}^-(D^2v,Dv) \leq 1 \text{ in } \{|Dv| \geq A_0+2\}\cap \tilde{\Omega}_{\tilde{\rho}+1}(\sigma e_n),\\
        \tilde{\beta} \cdot Dv \leq \eta \text{ on } \partial \tilde{\Omega}_{\tilde{\rho}+1}(\sigma e_n),\\
        |\{v>K\} \cap B_1(\sigma e_n)|> (1-\epsilon_0)|B_1|,
    \end{cases}
\end{align*}
where $\tilde{\Omega} = \frac{1}{r}(\Omega - \bar{x_0})$ and $\tilde{\beta}(x) = \beta(rx+\bar{x_0})$, which implies $\sup{|\tilde{\tau}|} \leq \sup|\tau| \leq \tau_0$.
Since $\varphi_{\tilde{\Omega}}(x) =\frac{\varphi_\Omega(rx+x'_0)-\varphi_\Omega(x'_0)}{r}$, we have $\varphi_{\tilde{\Omega}} (0) =0$ and $[\varphi_{\tilde{\Omega}}]_{C^1} \leq [\varphi_\Omega]_{C^1} \leq \mu$. Note that $B_1(\sigma e_n) \subset \tilde{\Omega}$.
Therefore using Corollary \ref{meas}, we conclude that $v>1$ in $\tilde{\Omega}_{\tilde{\rho}}(\sigma e_n)$ and so $u>K^m$ in $B_{\tilde{\rho}r}(x_0) \cap \Omega_1$.
In conclusion, we have $\tilde{\rho}B \cap \Omega_1 \subset F$.
Using that $\Omega_1$ satisfies $\rho$-corkscrew condition, we have $|E| \leq (1-3^{-n}\rho^n\epsilon_0 )|F|$ by Lemma \ref{ink} and it proves the claim.
Therefore, we get
\begin{align*}
    |\{ u > K^m\} \cap \Omega_1| \leq CK^{-m\epsilon}.
\end{align*}
where $-\epsilon = \log(1-3^{-n}\rho^n\epsilon_0 )/\log K$, which finishes the proof.
\end{proof}

\begin{lem}
    There exist small $\tau_0>0$, $\mu>0$, $\eta_1>0$, $\epsilon_1>0$, such that if $\sup|\tau| \leq \tau_0$ and $[\varphi_\Omega]_{C^1} < \mu$, then for any $q \in \mathbb{R}^n$, $r \leq 1$, $a \leq 1$, $0<\theta\leq\epsilon_1$ and $u \in C(\overline{\Omega_{\tilde{\rho}+1}})$ satisfying
    \begin{align*}
        \begin{cases}
            u \geq 0 \text{ in } \Omega_{(\tilde{\rho}+1)r}, \\
            \mathcal{P}^-(D^2u,Du) \leq \epsilon_1 \text{ in } \{|Du-q| \geq \theta\}\cap \Omega_{(\tilde{\rho}+1)r},\\
            \beta \cdot Du \leq \eta_1 \text{ on } \partial \Omega_{(\tilde{\rho}+1)r}, \\
            |\{u>r^a\} \cap \Omega_r| \geq \frac{1}{2}|\Omega_r|,
        \end{cases}
    \end{align*}
    then $u > \epsilon_1 r^a$ in $\Omega_r$.
\end{lem}

\begin{proof}
We choose $\theta_0 = 1$ and let $\epsilon,\eta,C_1$ be constants in Theorem \ref{Le}.
Since $\varphi_\Omega(0) =0$ and $[\varphi_\Omega]_{C^1} < \mu$, there exists a small $c=c(n,\mu)>0$ such that $c|B_r| \leq |\Omega_r|$.
Let $\kappa>1$ be a constant such that $C_1 \kappa^{-\epsilon} < \frac{c}{2}|B_1|$.
Consider $\tilde{u}(x) = \kappa r^{-a}u(rx)$. Then $\tilde{u}$ satisfies
\begin{align*}
    \begin{cases}
        \tilde{u} \geq 0 \text{ in } \tilde{\Omega}_{\tilde{\rho}+1}, \\
        \mathcal{P}^-(D^2\tilde{u},D\tilde{u}) \leq \kappa r^{2-a} \epsilon_1 \text{ in } \{|D\tilde{u}-\kappa r^{1-a}q| \geq \kappa r^{1-a}\theta\}\cap \tilde{\Omega}_{\tilde{\rho}+1},\\
        \tilde{\beta} \cdot D\tilde{u} \leq \kappa r^{1-a}\eta_1 \text{ on } \partial \tilde{\Omega}_{\tilde{\rho}+1},\\
        |\{\tilde{u} > \kappa \} \cap \tilde{\Omega}_1| \geq \frac{1}{2}|\tilde{\Omega}_1| \geq \frac{c}{2}|B_1| > C_1\kappa^{-\epsilon}, \\
    \end{cases}
\end{align*}
where $\tilde{\Omega} = \frac{1}{r} \Omega$ and $\tilde{\beta}(x) = \beta(rx)$.
Choosing $\epsilon_1 =\kappa^{-1}$ and $\eta_1 = \kappa^{-1}\eta $, we have $\kappa r^{2-a} \epsilon_1 \leq 1$, $\kappa r^{1-a} \theta \leq \theta_0$ and $\kappa r^{1-a} \eta_1 \leq \eta$.
Therefore, applying Theorem \ref{Le}, we get $\tilde{u} > 1$ in $\tilde{\Omega}_1$ and therefore $u >\epsilon_1 r^a$ in $\Omega_r$.
\end{proof}

Finally, by repeating the standard arguments in \cite{Imbert16}, we can prove the Theorem \ref{holder} and conclude that $u$ is H{\"o}lder continuous up to the boundary.

\section{Improvement of flatness} \label{sec4}

Now we prove a compactness result of oblique boundary condition as in \cite{Ricarte20}.
\begin{lem} \label{appro}
    Let $u$ satisfy $|u| \leq 1$ and be a viscosity solution of
    \begin{align*} 
        \begin{cases}
            |Du-q|^\gamma F(D^2u) =f \text{ in } \Omega_1,\\
            \beta \cdot Du = g \text{ on } \partial \Omega_1,
        \end{cases}
    \end{align*}
    where $q \in \mathbb{R}^n$.
    Given $\delta>0$, there exists $\epsilon=\epsilon(\delta,n, \lambda,\Lambda,\delta_0)>0$ such that if
    \begin{align*}
        \norm{f}_{L^\infty(\Omega_1)}, \norm{g}_{L^\infty(T_1)}, \norm{\beta-\beta_0}_{C^\alpha(T_1)}, \norm{\varphi}_{C^1(T_1)} \leq \epsilon,
    \end{align*}
    then there exists a function $h \in C^{1,\alpha_0}(\overline{B^+_{3/4}})$ such that
    \begin{align*}
        \begin{cases}
            F(D^2h) = 0 \text{ in } B^+_{3/4},\\
            \beta_0 \cdot Dh =0 \text{ on } T_{3/4},
        \end{cases}
    \end{align*}
    where $\beta_0 = \beta(0)$ is a constant vector and $\norm{u-h}_{L^\infty(\Omega_{1/2})} \leq \delta$.
\end{lem}
\begin{proof}
We assume the conclusion of the lemma is false.
Then there exist $\delta>0$ and a sequence of $F_k$, $u_k$, $f_k$, $g_k$, $\beta_k$, $\Omega_k$ and $q_k$ such that $u_k$ satisfies $|u_k| \leq 1$ and
\begin{align} \label{flat}
    \begin{cases}
        |Du_k -q_k|^\gamma F_k(D^2u_k)=f_k \text{ in } (\Omega_k)_1, \\
        \beta_k \cdot Du_k = g_k \text{ on } \partial (\Omega_k)_1,
    \end{cases}
\end{align}
with $\norm{f_k}_{L^\infty}$, $\norm{g_k}_{L^\infty}$, $\norm{\beta_k-\beta_k(0)}_{C^\alpha}$, $\norm{\varphi_{\Omega_k}}_{C^1} \leq \frac{1}{k}$, but $\norm{u_k-h_k}_{L^\infty((\Omega_k)_{1/2})} \geq \delta$ for any $h_k$ satisfying the corresponding conditions.
Notice that the equation \eqref{flat} has uniformly elliptic structure where $|Du_k - q_k| \geq 1$.
By the uniform boundary H{\"o}lder estimate in Theorem \ref{holder} and the Arzela-Ascoli theorem, a sequence $u_k$ converges to a function $u_{\infty} \in C(\overline{B^+_1})$ locally uniformly up to subsequence.
Observe that $F_k \rightarrow F_\infty$ uniformly in compact of $S(n)$ by the Arzela-Ascoli theorem.
Moreover, we have $\beta_k \rightarrow \beta_0$ uniformly for some constant oblique vector $\beta_0$ and $\Omega_k \rightarrow B^+_1$.
We claim that $u_\infty$ is a viscosity solution of
\begin{align*}
    \begin{cases}
        F_\infty(D^2u_\infty) = 0 \text{ in } B^+_1,\\
        \beta_0 \cdot Du_\infty =0 \text{ on }  T_1.
    \end{cases}
\end{align*}
If the claim is true, then consider a sequence of viscosity solutions $h_k$ satisfying
\begin{align*}
	\begin{cases}
		F_k(D^2h_k) = 0 \text{ in } B^+_{3/4},\\
		\beta_k(0) \cdot Dh_k =0 \text{ on }  T_{3/4},\\
		h_k = u_k \text{ on } \partial B^+_{3/4} \cap \{x_n>0\}.
	\end{cases}
\end{align*}
Existence of the solutions $h_k$ is guaranteed by Theorem 3.3 in \cite{Li18}.
By the Arzela-Ascoli theorem and the stability result, Proposition 2.1 in \cite{Li18}, a sequence $h_k$ converges to a function $h_{\infty} \in C(\overline{B^+_{3/4}})$ locally uniformly and $h_{\infty}$ satisfies
\begin{align*}
	\begin{cases}
		F_\infty(D^2h_\infty) = 0 \text{ in } B^+_{3/4},\\
		\beta_0 \cdot Dh_\infty =0 \text{ on }  T_{3/4},\\
		h_\infty = u_\infty \text{ on } \partial B^+_{3/4} \cap \{x_n>0\}.
	\end{cases}
\end{align*}
Since the solution of the above equation is unique by Theorem 3.3 in \cite{Li18}, we get $h_\infty = u_\infty$, which is a contradiction.

We now prove the claim. If the sequence $q_k$ is bounded, then we have subsequence such that $q_k \rightarrow q_\infty$ for some $q_\infty \in \mathbb{R}^n$.
Therefore, using stability results, we have that  $u_\infty$ is a viscosity solution of
\begin{align*}
    \begin{cases}
        |Du_\infty - q_\infty|^\gamma F(D^2u_\infty) = 0 \text{ in } B^+_1,\\
        \beta_0 \cdot Du_\infty =0 \text{ on } T_1.
    \end{cases}
\end{align*}
By the cutting lemma as in \cite{Imbert13,Ricarte20}, we have $F(D^2u_\infty) = 0 \text{ in } B^+_1$ and prove the claim.

If the sequence $q_k$ is unbounded, then we have
\begin{align*}
    \left|\frac{Du_k}{|q_k|} - \frac{q_k}{|q_k|}\right|^\gamma F(D^2u_k) = \frac{f_k}{|q_k|}.
\end{align*}
Since $\frac{q_k}{|q_k|} \rightarrow e$ for some unit vector $e$ and $|q_k| \rightarrow \infty$, $u_\infty$ is a viscosity solution of
\begin{align*}
    \begin{cases}
        |0\cdot Du_\infty -e|^\gamma F(D^2u_\infty) = 0 \text{ in } B^+_1,\\
        \beta_0 \cdot Du_\infty =0 \text{ on }  T_1.
    \end{cases}
\end{align*}
Since $|e|=1$, we also prove the claim.
\end{proof}

\begin{lem} \label{appro2}
    Let $u$ satisfy $|u| \leq 1$ and be a viscosity solution of
    \begin{align*}
        \begin{cases}
            |Du-q|^\gamma F(D^2u) =f \text{ in } \Omega_1,\\
            \beta \cdot Du = g \text{ on } \partial \Omega_1,
        \end{cases}
    \end{align*}
    where $q \in \mathbb{R}^n$.
    Given $\alpha <\min\left\{\alpha_0,\frac{1}{1+\gamma}\right\}$, there exist $0< r < 1/2$ and $\epsilon>0$ such that if
    \begin{align*}
        \norm{f}_{L^\infty(\Omega_1)}, \norm{g}_{L^\infty(T_1)}, \norm{\beta-\beta_0}_{C^\alpha(T_1)}, \norm{\varphi}_{C^1(T_1)} \leq \epsilon,
    \end{align*}
    then there exists an affine function $l = u(0)+b\cdot x$ such that
    \begin{align*}
        \norm{u-l}_{L^\infty(\Omega_{r})} \leq r^{1+\alpha},\\
        \beta_0 \cdot b = 0, \ |b| \leq C_e.
    \end{align*}
\end{lem}

\begin{proof}
For given $\delta>0$, let $\epsilon>0$ be the constant in Lemma \ref{appro}.
Then there exists a function $h \in C^{1,\alpha_0}$ which is a solution of a homogeneous equation and $\norm{u-h}_{L^\infty} \leq \delta$.
Thus from the $C^{1,\alpha_0}$ estimates for $h$, letting $\tilde{l}(x)= h(0) + Dh(0) \cdot x$, we have
\begin{align*}
    |h(x)-\tilde{l}(x)| \leq C_e|x|^{1+\alpha_0} \\
    \beta_0 \cdot Dh(0) = 0, \  |Dh(0)| \leq C_e.
\end{align*}
Now we choose small $r<\frac{1}{2}$ such that $C_er^{1+\alpha_0} \leq \frac{1}{3}r^{1+\alpha}$ and choose $\delta = \frac{1}{3}r^{1+\alpha}$.
Letting $l(x)=u(0)+ Dh(0) \cdot x$, we have
\begin{align*}
    \norm{u-\tilde{l}}_{L^\infty(\Omega_r)} &\leq \norm{u-h}_{L^\infty(\Omega_r)} + \norm{h-\tilde{l}}_{L^\infty(\Omega_r)} +|u(0) -h(0)| \\
    &\leq 2\delta + C_er^{1+\alpha_0} \leq r^{1+\alpha}.
\end{align*}
\end{proof}

\section{Proof of Theorem \ref{main}}
\label{sec5}
Now we are all set to give our complete proof of the
main result, Theorem \ref{main}.

\begin{proof}
Step 1. In this step we assert that we may assume $|u| \leq 1$,
$u(0) =0$ and $u$ satisfies
\begin{align*}
\begin{cases}
|Du-q|^\gamma F(D^2u) =f \text{ in } \Omega_1,\\
\beta \cdot Du = g \text{ on } \partial \Omega_1,
\end{cases}
\end{align*}
with $g(0) =0$, $\norm{f}_{L^\infty}\leq \epsilon$, $
\norm{g}_{C^\alpha}\leq \frac{\epsilon}{2}$, $
\norm{\beta-\beta_0}_{C^\alpha}\leq \frac{1-r^\alpha}{2C_e}\epsilon$,
$\varphi(0) = 0$, $D\varphi(0) =0$ and $\norm{\varphi}_{C^1} \leq
\epsilon$ where $r, \ \epsilon>0$ are from Lemma
\ref{appro2}.

Indeed, we first consider the scaled function $\tilde{u}(x) =
\frac{u(sx)}{K}$ with $0<s<1$ and $1<K<\infty$ to observe that
$\tilde{u}$ satisfies
\begin{align*}
\begin{cases}
|D\tilde{u}|^\gamma \tilde{F}(D^2\tilde{u}) = \tilde{f} \text{ in } \left(\frac{1}{s}\Omega\right) \cap B_1, \\
\tilde{\beta} \cdot D\tilde{u} = g_1 \text{ in } \partial
\left(\frac{1}{s}\Omega\right) \cap B_1,
\end{cases}
\end{align*}
where $\tilde{F}(M) = \frac{s^2}{K}F\left(\frac{K}{s^2}M\right)$,
$\tilde{f}(x) = \frac{s^{2+\gamma}}{K^{1+\gamma}}f(sx)$,
$\tilde{\beta}(x) = \beta(sx)$ and $g_1(x) =
\frac{s}{K} g(sx)$.

By choosing a proper coordinate, we assume that $\Omega_1$ can be
represented by the graph of $\varphi_\Omega$ with $\varphi_\Omega(0)
= 0$ and $D\varphi_\Omega(0) =0$. Since $\varphi_\Omega \in C^1$, there exists a small $s>0$ depending
on $C^1$ modulus of $\varphi_\Omega$ such that $|D\varphi_\Omega(x)|
\leq \epsilon$ for any $|x| \leq s$, where this $\epsilon$ was given
in Lemma \ref{appro2}. Note
$\tilde{\varphi}(x):=\varphi_{\frac{1}{s}\Omega}(x) =
\frac{\varphi(sx)}{s}$ to find $\norm{\tilde{\varphi}}_{C^1} \leq
\epsilon$. We also have
$\norm{\tilde{\beta}-\beta_0}_{C^\alpha} \leq s^\alpha
\norm{\beta-\beta_0}_{C^\alpha}$. Accordingly, we choose
$s<1$ such that
$\norm{\tilde{\varphi}}_{C^1}\leq \epsilon$ and
$\norm{\tilde{\beta}-\beta_0}_{C^\alpha} \leq
\frac{1-r^\alpha}{2C_e}\epsilon$.

Moreover, choosing $K = 4\left(1+\norm{u}_{L^\infty} +
\delta_0^{-1}\epsilon^{-1}\left(\norm{f}^{\frac{1}{1+\gamma}}_{L^\infty}
+ \norm{g}_{C^\alpha}\right)\right)$, we find $|\tilde{u}| \leq
\frac{1}{4}$, $\norm{\tilde{f}}_{L^\infty} \leq \epsilon$ and $\norm{
g_1}_{C^\alpha} \leq \frac{\delta_0}{4}\epsilon$.

Write $\tilde{v} (x) = \tilde{u}(x) - \frac{
 g_1 (0)}{\tilde{\beta}_n(0)}x_n-\tilde{u}(0)$ and $\tilde{q}=  - \frac{   g_1 (0)}{\tilde{\beta}_n(0)}e_n$, to discover that
$\tilde{v}$ satisfies
\begin{align*}
\begin{cases}
|D\tilde{v}-\tilde{q}|^\gamma  \tilde {F} ( D^2\tilde{ v }) = \tilde{f} \ & \text{ in } \tilde{ \Omega }_1 = \left(\frac{1}{s}\Omega\right) \cap B_1, \\
\tilde{\beta} \cdot D\tilde{v}= g_1 -\frac{ g_1
(0)}{\tilde{\beta}_n(0)}\tilde{\beta}_n =\tilde{g} & \text{ on }
\partial \tilde{ \Omega }_{1} = \partial \left(\frac{1}{s}\Omega\right) \cap B_1,
\end{cases}
\end{align*}
with $\tilde{v}(0) = 0$, $\tilde{g}(0)=0$, $|\tilde{v}| \leq 2|\tilde{u}|
+\frac{|
 g_1(0)|}{\delta_0} \leq 1$ and $\norm{\tilde{g}}_{C^\alpha} \leq
\norm{g_1}_{C^\alpha} + \frac{| g_1 (0)|}{\delta_0}
\norm{\tilde{\beta}_n}_{C^\alpha} \leq \frac{\epsilon}{2}$.

Step 2: This step claims that for $r, \ \epsilon>0$ from Lemma
\ref{appro2} and $k=0,1,...$, there exists $l_k(x) = b_k \cdot x$
such that
\begin{align}
\label{condlk}
\begin{cases}
\norm{u-l_k}_{L^\infty(\Omega_{r^k})} \leq r^{k(1+\alpha)}, \\
\beta_0 \cdot b_k =0, \\
|b_k - b_{k+1}| \leq C_er^{k\alpha}.
\end{cases}
\end{align}
This can be proved by induction on $k$. For $k=0$ it follows from
Step 1 with $l_0 = 0$.

Assuming that the claim is true for some $k$, we consider
\begin{align*}
v(x)=\frac{(u-l_k)(r^kx)}{r^{k(1+\alpha)}}.
\end{align*}
Then $|v| \leq 1$ in $\left(\frac{1}{r^k}\Omega\right) \cap B_1$, $v(0)=0$ and $v$ satisfies
\begin{align*}
\begin{cases}
|Dv+q_k|^\gamma F_k(D^2v) = f_k \text{ in } \left(\frac{1}{r^k}\Omega\right) \cap B_1, \\
\beta_k \cdot Dv = g_k \text{ in } \partial
\left(\frac{1}{r^k}\Omega\right) \cap B_1,
\end{cases}
\end{align*}
where
$$
F_k(M) = r^{k(\alpha-1)}F(r^{-k(\alpha-1)}M), \ q_k =
r^{-k\alpha}q-r^{-k\alpha}b_m,  \ f_k(x)=
r^{k(1-\alpha(1+\gamma))}f(r^kx)$$ and
$$ \beta_k(x) = \beta(r^kx), \ g_k(x) = r^{-k\alpha}g(r^kx) - r^{-k\alpha}\beta(r^kx) \cdot b_k = \tilde{g}_k(x) -
\tilde{\beta}_k(x).$$ We now check that $v$ satisfies the assumption
of Lemma \ref{appro2}. Since $\alpha(1+\gamma)\leq 1$, we have
\begin{align*}
\norm{f_k}_{L^\infty} \leq
r^{k(1-\alpha(1+\gamma))}\norm{f}_{L^\infty} \leq \epsilon.
\end{align*}
Also since $g(0)=0$, we get $\norm{\tilde{g}_k}_{L^\infty} \leq \norm{\tilde{g}}_{C^\alpha} \leq \frac{\epsilon}{2}$.
Using $\beta_0 \cdot b_k =0$, and $|b_k| \leq \frac{C_e}{1-r^\alpha}$, we have
\begin{align*}
\norm{\tilde{\beta}_k}_{L^\infty} = \norm{r^{-k\alpha}
(\beta(r^kx)-\beta_0) \cdot b_k}_{L^\infty} \leq |b_k|
\norm{\beta-\beta_0}_{C^\alpha} \leq \frac{\epsilon}{2},
\end{align*}
which implies $\norm{g_k}_{L^\infty} \leq \epsilon$. Observe also
that find $\norm{\beta_k-\beta_0}_{C^\alpha} \leq
\norm{\beta-\beta_0}_{C^\alpha} \leq \epsilon$. In addition, we have
$\varphi_k(x):= \varphi_{\frac{1}{r^k}\Omega}(x) =
\frac{\varphi(r^kx)}{r^k}$, and $\varphi_k(0)=0$, $D\varphi_k(0)=0$
and $\norm{\varphi_k}_{C^1} \leq \norm{\varphi}_{C^1} \leq
\epsilon$. Therefore, we are under the assumptions of Lemma
\ref{appro2}, which implies that there exists a linear function
$\tilde{l}(x) = \tilde{b}\cdot x$ such that
\begin{align*}
\norm{v-l}_{L^\infty(\frac{1}{r^k}\Omega \cap B_{r})} \leq
r^{1+\alpha}, \ \beta_0 \cdot \tilde{b} =0 \text{ and } |\tilde{b}|
\leq C_e.
\end{align*}
Thus $l_{k+1}(x) = l_k(x) + r^{k(1+\alpha)} \tilde{l}(r^{-k}x)$,
$l_{k+1}$ satisfies the condition \eqref{condlk} for $k+1$ and the
claim follows.

Step 3: Once we have reached here, we can use the standard argument
for the Schauder estimate, eventually to deduce that $l_\infty =
\lim_{k\rightarrow \infty} l_k$ is the affine approximation of $u$
at 0 and $u$ is $C^{1,\alpha}$ at 0. By the standard covering
argument we conclude that $u \in
C^{1,\alpha}\left(\overline{\Omega_{1/2}}\right)$.
\end{proof}

\bibliographystyle{amsplain}

\end{document}